\documentclass[12pt]{article}
\usepackage{amscd,amsfonts,amsmath,amssymb,amstext,amsthm}
\title {Some properties of a Subclass of Univalent Functions}
\date{}
\author{Ntatin B.\\
 \small Department of Mathematics\\
 \small Austin Peay State University\\
  \small Clarksville, TN 37043\\
  \small\texttt{ntatinb@apsu.edu}}
\newcommand{\s} {{\smallskip \noindent}}
\newcommand{\m} {{\medskip \noindent}}

\newtheorem{thm} {Theorem} [section]

\newtheorem{lem} [thm]{Lemma}

\theoremstyle{definition}

\theoremstyle{remark}
\newtheorem{rem} {Remark}
\begin{document}
\maketitle
\begin{abstract}
We give some coefficient bounds and distortion theorems for a subclass of univalent functions in the unit disk, and defined using the S\^{a}l\^{a}gean differential operator. The results generalize and unify some well known results for several subclasses of univalent functions defined on the unit disk having form $f(z)=z+\sum_{n=2}^{\infty}a_nz^n$, normalized such that $f(0)=0$ and $f^\prime(0)=1$.
\end{abstract}

\medbreak \noindent\medbreak \noindent{\ \textbf{Key Words:}}
coefficient bounds, S\^{a}l\^{a}gean differential operator, univalent functions

\setcounter{section}{0}

\section {Introduction and Preliminaries}
Let $\mathbb A$ denote the class of functions analytic in the unit disk $E={z\in \mathbb C:|z|<1}$, and of the form
$$
f(z)=z+\sum_{n=2}^{\infty}a_nz^n
$$
normalized such that $f(0)=0$ and $f^\prime(0)=1$.

Also, let ${\cal P}$ denote the class consisting of analytic functions $p(z)$, such that $p(0)=1$, with positive real part, i.e., $\Re\{p(z)\}>0$, in $E$. This class of functions is known as the class of Carath\'{e}odory functions and have form
$$
p(z)=1+\sum_{k=1}^{\infty}a_kz^k.
$$
There are many subclasses of univalent functions defined using similar conditions as in the case of the class ${\cal P}$ above, for example,
\begin{eqnarray*}
S_0&=&\{f(z)\in \mathbb A:\Re\{f(z)/z\}>0, z\in E\},\\
B(\beta)&=&\{f(z)\in\mathbb A:\Re\{f(z)/z\}>\beta, z\in E, 0\leq\beta<1\},\\
\delta(\beta)&=&\{f(z)\in\mathbb A:\Re\{f^\prime(z)\}>\beta, z\in E, 0\leq\beta<1\},\\
B_n(\alpha)&=&\{f(z)\in\mathbb A:\Re\{D^n[f(z)^\alpha/z^\alpha]\}>0, z\in E, \alpha>0, n=0,1,2,...\}.
\end{eqnarray*}

All these subclasses of functions analytic in the unit disk $E$ are well known and have been investigated repeatedly by several authors including (for example, \cite{AO}, \cite{Ha}, \cite{JI}, \cite{SG}, \cite{Sri}, and \cite{TD}). The subclass $B_n(\alpha)$ is called the class of Bazilevi\^{c} functions of type $\alpha$ and $D^n$ is the S\^{a}l\^{a}gean differential operator defined recursively as follows \cite{SG1}; $D^0f(z)=f(z), D^1f(z)=Df(z)=zf(z)$, and $D^nf(z)=D(D^{n-1}f(z))=z[D^{n-1}f(z)]^\prime$, for $n\geq 2$.
\smallskip

We define the class $T^\alpha_n(\beta)$ to be the subclass of $\mathbb A$ consisting of functions satisfying the conditions;
$$
\Re\{D^n[f(z)^\alpha/z^\alpha]\}>\beta, z\in E, \alpha>0, n=0,1,2,\ldots,
$$
where as in the definition of $B_n(\alpha)$ above, $0\leq \beta<1$, and $D^n$ is the S\^{a}l\^{a}gean differential operator.
\smallskip

There are certain relations between some of the subclasses above, for example, $\delta(0)=B_1(1)$. In particular, $T^\alpha_n(\beta)$ is a generalization of some kind of all the subclasses of functions defined above. Indeed, if in the definition of $T^\alpha_n(\beta)$ we put
\begin{itemize}
\item $n=0, \alpha=1, \beta=0$, we obtain the class $S_0$.
\item $n=0, \alpha=1$, we obtain the class $B(\beta)$.
\item $n=1, \alpha=1, \beta=0$, we obtain the class $\delta(0)=B_1(1)$.
\item $n=0, \alpha=1$, we obtain the class $\delta(\beta)$.
\end{itemize}

In \cite{TO}, Opoola T.O. established the following properties for functions in the class $T^\alpha_n(\beta)$;
\begin{itemize}
\item for $n\geq 1, T^\alpha_n(\beta)\subset \mathbb A$,
\item for $n\geq1, T^\alpha_{n+1}(\beta)\subset T^\alpha_n(\beta/\alpha)$,
\item if $f(z)\in T^\alpha_n(\beta)$, then $\Re[f(z)/z]^\alpha >\beta$, for $z\in E$,
\item if $f(z)\in T^\alpha_n(\beta)$ and $\alpha+c > 0$, then the function defined by
$$
F(z)^\alpha = \frac{\alpha+c}{z^c}\int t^{c-1}f(t)dt, z\in E,
$$
is also in the class $T^\alpha_n(\beta)$.
\end{itemize}
\m
\begin{rem}
Observe that for $f(z)=z+a_2z^2+\ldots$, we have that
$$
D[f^\alpha(z)]=z[f^\alpha(z)]^\prime =\alpha z^\alpha+\alpha a_2z^{\alpha+1}+\ldots.
$$
Now since $D^{n+1}[f^\alpha(z)]=D^n[Df^\alpha(z)]$, it follows that if $f(z)\in T^\alpha_{n+1}(\beta)$, that is, $\Re\{D^{n+1}[f^\alpha(z)]/z^\alpha\}$, then $f(z)\in T^\alpha_n(\beta/\alpha)$, or equivalently, $\Re\{D^n[f^\alpha(z)]/(\alpha z^\alpha)]\}>\beta$, and the second property above follows.

\s
It is worth mentioning here that the second property above was actually misstated in \cite{TO} as follows; $T^\alpha_{n-1}(\beta)\subset T^\alpha_n(\beta)$ for $n\geq 1$.
\end{rem}

\m
Our main aim in the present paper is to further investigate the properties of functions in this subclass of univalent functions. In particular, for functions in the subclass $T^\alpha_n(\beta)$, we obtain some coefficient bounds, and a distortion theorem (see Theorems 2.1 and 3.1 below).

\section {Coefficient Bounds}
In general, the coefficient problem for functions in class $\mathbb A$ is to determine the region of $\mathbb C^{n-1}$, the $(n-1)-$dimensional complex plane, occupied by the points $\{a_2,a_3,\ldots ,a_n\}$ for all functions $f(z)\in\mathbb A$. However, the deduction of such precise analytic information from the singular geometric hypothesis of univalence, is seemingly very difficult. We therefore, confine attention to the more specific problem of estimating $|a_n|$ instead.
For the subclass $T^\alpha_n(\beta)$ of the class $\mathbb A$, we establish the following coefficient bounds:

\begin{thm}
Let
$$
f(z)=z+\sum_{n=2}^{\infty}a_nz^n \in T^\alpha_n(\beta),
$$
then for $\beta \in [0,1]$ and $z\in E$, we have the following estimates;

\begin{enumerate}
\item $|a_2|\leq \frac{2(1-\beta)\alpha^{n-1}}{(\alpha+1)^n}$, for $\alpha >0$.

\item $|a_3|\leq \begin{cases} \frac{2(1-\beta)\alpha^{n-1}[(\alpha+1)^{2n}-(\alpha-1)\alpha^{n-1}(1-\beta)(\alpha-2)]}{(\alpha+2)^2(\alpha+1)^{2n}}, &\mbox{for}\; 0<\alpha\leq 1,\\
\frac{2(1-\beta)\alpha^{n-1}}{(\alpha+2)^n}, &\mbox{for} \;\alpha\geq 1.
\end{cases}$
\item $|a_4|\leq \begin{cases} \frac{A_1-A_2+A_3-A_4}{3(\alpha+3)^n(\alpha+2)^n(\alpha+1)^{3n}}, &\mbox{for}\; 0<\alpha\leq 1,\\
\frac {2(1-\beta)\alpha^{n-1}}{(\alpha+3)^n}, &\mbox{for}\; \alpha \geq 1,
\end{cases}$
\end{enumerate}
where
\begin{itemize}
\item $A_1=6\alpha^{n-1}(1-\beta)(\alpha+2)^n(\alpha+1)^{3n}$.\\
\item $A_2=12(\alpha-1)\alpha^{2n-2}(1-\beta)(\alpha+1)^{2n}(\alpha+3)^n$.\\
\item $A_3=12(\alpha-1)^2\alpha^{3n-3}(1-\beta)(\alpha+2)^n(\alpha+3)^n$.\\
\item $A_4=2(\alpha-1)(\alpha-2)\alpha^{2n-2}(1-\beta)(\alpha+2)^3(\alpha+3)^n$.
\end{itemize}
Inequality $(1)$ is sharp and equality occurs for the Koebe function
$$
K(z)=\frac{z}{(1-z)^2}
$$
or one of its rotations $\tilde{K}(z)=\exp(-\xi i)K(\exp(\xi i)z)$, where $\xi$ is a real number.
\end{thm}
\begin{proof}
Recall that by definition, $f(z)=z+\sum_{n=2}^{\infty}a_nz^n \in T^\alpha_n(\beta)$ if and only if $\Re\{D^n[f(z)^\alpha/z^\alpha]\}>\beta, \mbox{for} \; z\in E, \alpha>0,\;0\leq\beta\leq 1,\; \mbox{and}\;n=0,1,2,\ldots$. Now using binomial expansion, we obtain
\begin{equation}\label{1}
\begin{split}
f(z)^\alpha  = &z^\alpha+\alpha a_2z^{\alpha+1}+[\alpha a_3+\frac{\alpha(\alpha-1)}{2!}a_2^3]z^{\alpha+2}\\
           + &[\alpha a_4+\frac{\alpha(\alpha-1)}{2!}2a_2a_3+\frac{\alpha(\alpha-1)(\alpha-2)}{3!}a_2^3]z^{\alpha+3}\\
           + &[\alpha a_5+\frac{\alpha(\alpha-1)}{2!}(2a_2a_4+a_3^2)+\frac{\alpha(\alpha-1)(\alpha-2)}{3!}3a_2^2a_3\\
           + &\frac{\alpha(\alpha-1)(\alpha-2)(\alpha-3)}{4!}a_2^4]z^{\alpha+4}+\ldots.
\end{split}
\end{equation}
With hindsight of the above expansion, we define a new function as follows;
\begin{equation}\label{2}
\begin{split}
P(z) = &\frac {D^nf(z^\alpha)/(\alpha^nz^\alpha)}{1-\beta}\\
     =& 1+\frac{\alpha(\alpha+1)^na_2}{\alpha^n(1-\beta)}z+\frac{\alpha(\alpha+2)^n}{\alpha^n(1-\beta)}[a_3+\frac{(\alpha-1)}{2!}a_2^2]z^2\\
     &+ \frac{\alpha(\alpha+3)^n}{\alpha^n(1-\beta)}[a_4+\frac{(\alpha-1)}{2!}2a_2a_4+\frac{(\alpha-1)(\alpha-2)}{3!}a_2^3]z^3\\
    &+ \frac{\alpha(\alpha+4)^n}{\alpha^n(1-\beta)}[a_5+\frac{(\alpha-1)}{2!}(2a_2a_4+a_2^2)\\
     &+\frac{(\alpha-1)(\alpha-2)}{3!}3a_2^2a_3+\frac{(\alpha-1)(\alpha-2)(\alpha-3)a_2^4}{4!}]z^4\\
    =: &1+c_1z+c_2z^2+c_3z^3+c_4z^4\ldots
\end{split}
\end{equation}
Observe that because of the normalization on $f(z)$, that is, $f(z)=0$ and $f^\prime(z)=1$, it follows that $P(z)$ is analytic in the unit disk $E$, with positive real part, $\Re\{P(z)\}>0$. As a consequence, we conclude that $P(z)\in {\cal P}$, and it follows that for each $i=1,2,\ldots$, we have that, (see \cite{GA}).
\begin{equation}\label{3}
|c_i|\leq 2
\end{equation}
Now, comparing coefficients in equation(\ref{2}) above we get that

\begin{eqnarray}
c_1&=&\frac{(\alpha+1)^n}{\alpha^{n-1}(1-\beta)}a_2,\\
c_2&=&\frac{(\alpha+2)^n}{\alpha^{n-1}(1-\beta)}[a_3+\frac{(\alpha-1)}{2!}a_2^2],\\
c_3&=&\frac{(\alpha+3)^3}{\alpha^{n-1}(1-\beta)}[a_4+\frac{(\alpha-1)}{2!}2a_2a_3+\frac{(\alpha-1)(\alpha-2)}{3!}a_2^3].
\end{eqnarray}
By applying inequality $(3)$ to equation $(4)$ above, we obtain
$$
|a_2|\leq \frac{2(1-\beta)\alpha^{n-1}}{(\alpha+1)^n},
$$
which is inequality (1) in Theorem 2.1.
Substituting $a_2$ in equation $(5)$ above and simplifying gives
$$
a_3=\frac{\alpha^{n-1}(1-\beta)}{(\alpha+2)^n}c_2 -\frac{(\alpha-1)\alpha^{2n-2}(1-\beta)^2}{2(\alpha+1)^{2n}}c_1^2.
$$
Observe that in the second term in the above expression for $a_3$, every factor is a square, hence positive, except for $(\alpha-1)$ which is negative for $0\leq \alpha <1$, and zero for $\alpha =1$. Consequently, we distinguish two cases; the case when $0\leq \alpha\leq 1$ and the case when $\alpha \geq 1$. Again, applying inequality $(3)$ in each of these cases yields

\begin{equation*}
\begin{split}
|a_3|\leq &\frac{2\alpha^{n-1}(1-\beta)}{(\alpha+2)^n}-\frac{2(\alpha-1)\alpha^{2n-2}(1-\beta^2)}{(\alpha+1)^{2n}}\\
         = &2\alpha^{n-1}(1-\beta)[\frac{1}{(\alpha+2)^n}-\frac{(\alpha-1)\alpha^{n-1}(1-\beta)}{(\alpha+1)^{2n}}]\\
         = &2\alpha^{n-1}(1-\beta)[\frac{(\alpha+1)^{2n}-(\alpha-1)\alpha^{n-1}(1-\beta)(\alpha-2)}{(\alpha+2)^n(\alpha+1)^{2n}}]\\
           &\mbox{for}\; 0\leq \alpha\leq 1,
\end{split}
\end{equation*}
and
$$
|a_3|\leq \frac{2(1-\beta)\alpha^{n-1}}{(\alpha+2)^n},\; \mbox{for} \; \alpha\geq 1,
$$
which are the first and second parts respectively in inequality (2) of Theorem 2.1.

Now, substituting $a_2$ and $a_3$ into equation (6) and simplifying yields

\begin{equation}
\begin{split}
a_4=&\frac{\alpha^{n-1}(1-\beta)c_3}{(\alpha+3)^n} - \frac{(\alpha-1)\alpha^{n-1}(1-\beta)}{(\alpha+1)^n}c_1[\frac{\alpha^{n-1}(1-\beta)}{(\alpha+2)^n}c_2\\
      &-\frac{(\alpha-1)\alpha^{2n-2}(1-\beta)^2}{2(\alpha+1)^{2n}}c_1^2]-\frac{(\alpha-1)(\alpha-2)}{3!}.\frac{\alpha^{2n-2}(1-\beta)^2}{(\alpha+1)^{2n}}c_1^2\,.
\end{split}
\end{equation}

Once more, we consider two cases; firstly, suppose that $0\leq\alpha\leq 1$, then applying inequality $(3)$ to the above expression for $a_4$ (equation (7)), and simplifying gives the following estimate;
\begin{equation}
\begin{split}
|a_4|\leq &\frac{2\alpha^{n-1}(1-\beta)}{(\alpha+3)^n}-\frac{2(\alpha-1)\alpha^{n-1}(1-\beta)}{(\alpha+1)^n}[\frac{2\alpha^{n-1}(1-\beta)}{(\alpha+2)^n}\\
          &-\frac{2(\alpha-1)\alpha^{2n-2}(1-\beta)^2}{(\alpha+1)^{2n}}]-\frac{2(\alpha-1)(\alpha-2)\alpha^{2n-2}(1-\beta)^2}{3(\alpha+1)^{2n}}\\
         =&\frac{2\alpha^{n-1}(1-\beta)}{(\alpha+3)^n}-\frac{4(\alpha-1)\alpha^{2n-2}(1-\beta)^2}{(\alpha+1)^n(\alpha+2)^n}\\
          &=\frac{4(\alpha-1)^2\alpha^{3n-3}(1-\beta)^3}{(\alpha+1)^{2n}}]-\frac{2(\alpha-1)(\alpha-2)\alpha^{2n-2}(1-\beta)^2}{3(\alpha+1)^{3n}}\\
         =&\frac{A_1-A_2+A_3-A_4}{3(\alpha+2)^n(\alpha+3)^n(\alpha+1)^{3n}}\,,
\end{split}
\end{equation}
where the $A_i's$ are as specified in Theorem 2.1.
Next, observe that $(\alpha -1)$ will be negative if $\alpha\geq 1$, hence the second and third terms in inequality (8) will both be negative as such, thus the second part of inequality (3) of Theorem 2.1 follows. That is
$$
|a_4|\leq\frac{2(1-\beta)\alpha^{n-1}}{(\alpha+3)^3},\, \mbox{for}\,\alpha\geq 1.
$$
Finally, equality is attained in inequality (1) of Theorem 2.1 for the Koebe function
$$
k(z)=\frac{z}{(1-z)^2}=z+2z^2+3z^3+\ldots,
$$
since for this function,
$$
\frac{D^n[k(z)^\alpha]/(\alpha^nz^n)-\beta}{(1-\beta)}=1+\frac{2\alpha(\alpha+1)^n}{\alpha^n(1-\beta)}+\ldots.
$$
\end{proof}

\begin{rem}
Theorem 2.1 in fact generalizes and unifies certain known results involving coefficient bounds for a good number of subclasses of univalent functions. In particular, if we take $n=1$ and $\beta=0$ in the definition of $T_n^\alpha(\beta)$, our results tally with those obtained for the class $B_1(\alpha)$, (see \cite{RS}).
Furthermore, if we let $\alpha=0$, respectively, $\alpha=1$, Theorem 2.1 gives the well-known coefficient inequalities for the class $S^*$, (the class of starlike functions with respect to the origin), respectively, ${\cal P}^\prime$, (class consisting of derivatives of functions in class ${\cal P}$) \cite{GA}.
\end{rem}

\section{Estimation of the functional $|a_3-\mu a_2^2|$ in $T_n^\alpha(\beta)$}

Intimately related with the coefficient problem is the problem of estimating the functional $|a_3-\mu a_2^2|$, where $\mu$ is a real parameter. In this section, we give estimates for this functional in the class $T_n^\alpha(\beta)$.

\m
We recall from the proof of Theorem 2.1 that the coefficients $a_2$ and $a_3$ were given by the expressions
\begin{eqnarray*}
a_2&=&\frac{c_1(1-\beta)\alpha^{n-1}}{(\alpha+1)^n}~\; \mbox{and}\,\\
a_3&=&\frac{c_2(1-\beta)\alpha^{n-1}}{(\alpha+2)^n}-\frac{c_1^2(1-\alpha)(1-\beta)^2\alpha^{2n-2}}{2(\alpha+1)^{2n}}\,,
\end{eqnarray*}
where $c_1$ and $c_2$ satisfy the inequality (3),
$$
|c_i|\leq 2~\, \mbox{for}~\, i=1,2,
$$
since they are the coefficients of a function belonging to the class ${\cal P}$.
Now, for $\mu$ a real number, we have that
\begin{eqnarray*}
a_3-\mu a_2^2 &=&\frac{c_2(1-\beta)\alpha^{n-1}}{(\alpha+2)^n}-\frac{c_1^2(\alpha-1)(1-\beta)^2\alpha^{2n-2}}{2(\alpha+1)^{2n}}+\frac{\mu c_1^2(1-\beta)^2\alpha^{2n-2}}{(\alpha+1)^{2n} }\\
           &=&\frac{c_2\alpha^{n-1}(1-\beta)}{(\alpha+2)^n}-\frac{c_1^2(\alpha-1)\alpha^{2n-2}(1-\beta)^2}{2(\alpha+1)^{2n}}[(\alpha-1)-2\mu]\,.
\end{eqnarray*}
Using inequality (3) together with the triangular inequality, we obtain
\begin{equation}
|a_3-\mu a_2^2|\leq \frac{2\alpha^{n-1}(1-\beta)}{(\alpha+2)^n}+\frac{2(\alpha-1)\alpha^{2n-2}(1-\beta)^2}{(\alpha+1)^{2n}}|2\mu-(\alpha-1)|\,.
\end{equation}
Considering the absolute value in the second term in the right hand side of inequality (3), we distinguish two cases: on the one hand, if $2\mu-(\alpha-1)\leq 0$, then $\mu \leq1/2(\alpha-1)$ and the second term in inequality (9) becomes negative, and so we obtain the estimate
$$
|a_3-\mu a_2^2|\leq \frac{2\alpha^{n-1}(1-\beta)}{(\alpha+2)^n}~\; \mbox{if}~\; \mu\leq1/2(\alpha-1)\,.
$$

On the other hand, if $2\mu-(\alpha-1)\geq 0$, then $\mu\geq 1/2(\alpha-1)$, consequently, the second term in the right hand side of inequality (9) is positive and this time we obtain the estimate
$$
|a_3-\mu a_2^2|\leq\frac{2\alpha^{n-1}(1-\beta)}{(\alpha+2)^n}+\frac{2(\alpha-1)\alpha^{2n-2}(1-\beta)^2}{(\alpha+1)^{2n}}[2\mu-(\alpha-1)]\,.
$$
We have thus proved the following
\begin{thm}
Let $f(z)=z+\sum_{k=2}^\infty a_kz^k \in T_n^\alpha(\beta)$, then for any real parameter $\mu\in\mathbb R$,
\begin{eqnarray*}
|a_3-\mu a_2^2|&\leq&\begin{cases}\frac{2\alpha^{n-1}(1-\beta)}{(\alpha+2)^n}~,\;\mbox{if}~\; \mu\leq\frac{\alpha-1}{2}\,,\\
                  \frac{2\alpha^{n-1}(1-\beta)}{(\alpha+2)^n}+\frac{2(\alpha-1)\alpha^{2n-2}(1-\beta)^2}{(\alpha+1)^{2n}}[2\mu-(\alpha-1)]~,\;\mbox{if}~\;\mu\geq\frac{\alpha-1}{2}\,.
\end{cases}
\end{eqnarray*}
\end{thm}

\section{Distortion Theorem for functions in $T_n^\alpha(\beta)$}
Many authors have turned attention to the so called coefficient estimate problems for different subclasses of univalent functions. Closely related to this problem is the problem of determining how large the modulus of a univalent function together with its derivatives can be in a particular subclass. Such results, referred to as distortion theorems, provide important information about the geometry of functions in these subclasses. In this section, we give a distortion theorem for functions in $T_n^\alpha(\beta)$.

\begin{thm}
Let $f(z)\in T_n^\alpha)\beta)$ and $z=re^{i\theta}$, where $0\leq \theta\leq 2\pi$, then
$$
\frac{(r-1)^2-\alpha^n(1+r)^2}{2r(1+r)}\leq\Re\left\{\frac{D^nf^\alpha(z)}{z^\alpha}\right\}\leq\alpha^n\frac{1+r}{1-r}\,.
$$
\end{thm}
In order to prove this theorem, we will need the following result which gives a representation of functions in class $T_n^\alpha(\beta)$ in terms of an analytic function in the unit disk.

\begin{lem}
The function $f(z)$, $z\in E$ belongs to the class $T_n^\alpha(\beta)$ if and only if there exist a function $\phi(z)$, analytic in $E$ with the property that $|\phi(z)|\leq1$ such that
$$
\frac{D^nf^\alpha(z)}{z^\alpha}=(2\beta-1)+\frac{2(1-\beta)}{1+z\phi(z)}\,.
$$
\end{lem}

\begin{proof}
Set $A(z)=\frac{D^nf^\alpha}{z^\alpha}-\beta$ and define
$$
B(z)=\frac{A(z)-(1-\beta)}{A(z)+(1-\beta)}=\frac{D^nf(z)/z^\alpha-1}{D^nf^\alpha(z)/z^\alpha-(2\beta-1)}\,.
$$
Clearly, the function $B(z)$ ia analytic for $|z|<1$ and since by the normalization on $f(z)$, that is, $f(0)=0$ and $f^\prime(0)=1$, it follows that $B(0)=0$. Consequently, we have that $|B(z)|<1$ for $|z|<1$, and so by Schwarz's lemma, we get that $|B(z)|<|z|$ for $|z|<1$. That is,
$$
\left|\frac{D^nf(z)/z^\alpha-1}{D^nf^\alpha(z)/z^\alpha-(2\beta-1)}\right|<|z|~ \; \mbox{for} \; |z|<1\,.
$$
Equivalently, we get that

\begin{equation}
\frac{D^nf(z)/z^\alpha-1}{D^nf^\alpha(z)/z^\alpha-(2\beta-1)}=z\phi(z)
\end{equation}
where $\phi(z)$ is an analytic function with the property that $|\phi(z)|\leq 1$ for $|z|<1$. Solving equation (1) for $\frac{D^nf^\alpha}{z^\alpha}$, we obtain

\begin{equation}
\frac{D^nf^\alpha(z)}{z^\alpha}=\frac{1-(2\beta-1)z\phi(z)}{1-z\phi(z)}=(2\beta-1)+\frac{2(1-\beta)}{1+z\phi(z)}\,.
\end{equation}

\s
Conversely, suppose that equation (2) holds, then by letting $z$ be real and approach $1$ from below, i.e., $z\rightarrow 1^-$, we get that
$$
\frac{D^nf^\alpha(z)}{z^\alpha}>(2\beta-1)+\frac{2(1-\beta)}{1+1.1}=2\beta-1+1-\beta=\beta\,.
$$
Hence $f(z)\in T_n^\alpha(\beta)$, completing the proof.

\end{proof}

{\it Proof of Theorem 1.1}. Imploring Lemma 1.2, we have that $f(z)\in T_n^\alpha(\beta)$ if and only if

\begin{equation}
\frac{D^nf^\alpha(z)}{z^\alpha}=(2\beta-1)+\frac{2(1-\beta)}{1+z\phi(z)}=:\alpha^nP(z)\,,
\end{equation}
where $P(z)\in{\cal P}$ is a function with positive real part. It is known that for such functions, (see for e.g.,[?]),

$$
\Re\{P(z)\}\leq\frac{1+r}{1-r}\, \mbox{where}\; |z|<1\,.
$$
On the other hand, taking real parts in equation $(3)$ gives

\begin{eqnarray*}
\Re\{P(z)\}&=&\frac{(2\beta-1)}{\alpha^n}+\frac{2(1-\beta)}{\alpha^n}\Re\left\{\frac{1}{1+z\phi(z)}\right\}\\
       &\leq &\frac{(2\beta-1)}{\alpha^n}+\frac{2(1-\beta)}{\alpha^n}.\frac{1}{1+r|\phi(z)|}\\
       &\leq &\frac{(2\beta-1)}{\alpha^n}+\frac{2(1-\beta)}{\alpha^n(1+r)}\; \; \mbox{since}\; \; |\phi(z)|<1\,.
\end{eqnarray*}
It therefore follows that
$$
\frac{1+r}{1-r}\geq \frac{(2\beta-1)}{\alpha^n}+\frac{2(1-\beta)}{\alpha^n(1+r)}\,,
$$
and solving for $\beta$ we obtain
$$
\beta\geq \frac{(r-1)^2-\alpha^n(1+r)}{2r(1+r)}\,,
$$
from which we conclude that
$$
\Re\left\{\frac{D^nf^\alpha(z)}{z_\alpha}\right\}\geq \frac{(2\beta-1)}{\alpha^n}+\frac{2(1-\beta)}{\alpha^n(1+r)}\,,
$$
which in fact is the left hand side of the inequality of theorem 1.1. The right hand side is straight forward since
$$
\Re\left\{\frac{D^nf^\alpha(z)}{z_\alpha}\right\}=\alpha^n\Re\{P(z)\}\leq\alpha^n\frac{(1+r)}{(1-r)}\,.
$$
This completes the proof of Theorem 1.1. \qed

\end{document}